%% file: preprint.tex
\documentclass{article}

\input{packages}
\usepackage{authblk}

\input{newcommands}

\newcommand{\Paragraph}[1]{\paragraph{#1}}

\begin{document}

\title{\thetitle}

\author[1]{\firstauthor}
\author[2]{\secondauthor}
\author[1]{\thirdauthor}

\affil[1]{\firstaffil}
\affil[2]{\secondaffil}

\maketitle

\begin{abstract}
  \input{abstract}
\end{abstract}

\input{doc}

\end{document}

%% file: packages.tex
\usepackage[backend=biber,maxbibnames=5]{biblatex}

\usepackage{xspace}
\usepackage{siunitx}
\usepackage{csvsimple}
\usepackage{todonotes}
\usepackage{mathtools,dsfont}
\usepackage{amsmath,amssymb,amsthm}

\usepackage[algo2e,norelsize,ruled,vlined,commentsnumbered]{algorithm2e}

\usepackage{tikz}
\usetikzlibrary{matrix, positioning, patterns}
\usepackage{pgfplots}
\pgfplotsset{compat=1.18}
\usepackage{tikzscale}
\usepackage{adjustbox}
\usepackage{subcaption}
\usepackage[group-digits=integer,group-minimum-digits=4,print-unity-mantissa=false]{siunitx}
\usepackage{booktabs}
\usepackage{multicol}
\usepackage{placeins}

\usetikzlibrary{external}

\graphicspath{{Figures/}}

\usepackage[hidelinks]{hyperref}

%% file: newcommands.tex
\newcommand{\st}{\textrm{s.t.\xspace}}
\newcommand{\ie}{i.e.,\xspace}

\newcommand{\define}{\coloneqq}

\newcommand{\Real}{\mathds{R}}
\newcommand{\Nat}{\mathds{N}}

\newcommand{\Lag}{\mathcal{L}}

\newcommand{\IndexSet}{\mathcal{I}}
\newcommand{\idxVC}[1]{#1+l}
\newcommand{\idxCC}[1]{#1}

\newcommand{\primalProd}[2]{\langle #1, #2 \rangle_{X}}
\newcommand{\dualProd}[2]{\langle #1, #2 \rangle_{Y}}

\newcommand{\primalNorm}[1]{\|#1\|_{X}}
\newcommand{\dualNorm}[1]{\|#1\|_{Y}}
\newcommand{\matrixScalarprod}[2]{(#1)^T #2 (#1)}

\theoremstyle{plain}
\newtheorem{theorem}{Theorem}[section]

\theoremstyle{definition}

\DeclareMathOperator{\MPVC}{MPVC}
\DeclareMathOperator{\MPVCS}{MPVCS}

\DeclareMathOperator{\argmin}{arg\,min}

\DeclareMathOperator{\inc}{inc}

\DeclareMathOperator{\init}{init}

\newcommand{\signature}{\sigma}

\renewcommand{\epsilon}{\varepsilon}

\newcommand{\instName}[1]{\texttt{#1}}

\newcommand{\thetitle}{A Flow-based Method for Problems with Vanishing Constraints}

\newcommand{\firstauthor}{Christoph Hansknecht}
\newcommand{\secondauthor}{Julian Niederer}
\newcommand{\thirdauthor}{Andreas Potschka}

\newcommand{\firstaffil}{Institute of Mathematics, Clausthal University of Technology, Clausthal-Zellerfeld, Germany}
\newcommand{\secondaffil}{Interdisciplinary Center for Scientific Computing, University of Heidelberg, Heidelberg, Germany}

\hypersetup{
  pdftitle = {\thetitle},
  pdfauthor = {\firstauthor, \secondauthor, \thirdauthor}
}

\tikzset{
  potential_bar/.style={line cap=round, line join=round, very thin},
  bar/.style={line cap=round, line join=round},
  potential_coord/.style={fill,circle, inner sep=0},
  loading_force/.style={->,>=stealth, very thick, gray},
  solution_bar/.style={line join=round},
}

%% file: abstract.tex
Mathematical Programs with Vanishing Constraints (MPVCs) are a
notoriously challenging class of problems owing to their lack of
constraint qualification. Therefore, to tackle these problems,
relaxation-based approaches are typically used. While often yielding
satisfactory results, they generally require significant manual tuning
and adjustment of the relaxation parameter. To circumvent these
problems, we introduce a novel approach based on piecewise gradient
flows leading to first-order stationary points. We demonstrate
the effectiveness of our method on several real-world MPVC instances
and compare it to a common relaxation approach.

%% file: doc.tex
\section{Introduction}

This paper investigates the usefulness of the \emph{Sequential Homotopy Method}~\cite{sequential_homotopy}
for solving so-called \emph{Mathematical Programs with Vanishing Constraints}~\cite{mpvc}. These programs, called MPVCs for short are
optimization problems of the form
\begin{equation}
  \label{eq:mpvc}
  \tag{MPVC}
  \begin{aligned}
    \min_{x \in C} \quad & f(x) \\
    \st \quad & g(x) = 0, && \\
               & H_{i}(x) \geq 0, && \text{for } i = 1, \ldots , l, \\
               & G_{i}(x)H_{i}(x) \geq 0, && \text{for } i = 1, \ldots , l \\
  \end{aligned}
\end{equation}
for functions $f, G_{i}, H_{i}: \Real^{n} \to \Real$, and $g : \Real^{n} \to \Real^{m}$, which we
assume to be at least twice continuously differentiable, 
$C \define \{x \in \Real^{n} \mid x_\mathrm{l} \leq x \leq x_\mathrm{u} \}$ the convex set derived from the (possibly infinite) variable bounds on $x$.

The difficulty in solving MPVCs arises from to the products $G_{i}(x)H_{i}(x)$.
The constraints
$G_{i}$ are called \emph{vanishing} because, depending on the value of their corresponding \emph{controlling} constraint $H_{i}$, $G_{i}$ is either present in the problem (when $H_{i}(x) > 0$) or vanishes from it (when $H_{i}(x) = 0$).
Consequently the feasible region of an MPVC decomposes into different 
pieces where either $H_{i}(x) > 0$ and $G_{i}(x) \geq 0$ or $H_{i}(x) = 0$
and $G_{i}(x)$ is free. The intersection of these pieces (if it exists)
consists of the points $x$ satisfying $G_{i}(x) \geq 0$ and $H_{i}(x) = 0$. Following the convention in~\cite{mpvc}, we define for a feasible point $x \in \Real^{n}$ the following
index sets
\begin{equation*}
  \begin{aligned}
    \begin{aligned}[t]
        \IndexSet_{00}(x) &\define \{i \mid H_{i}(x) = 0, G_{i}(x) = 0\}, \\
  \IndexSet_{0+}(x) &\define \{i \mid H_{i}(x) = 0, G_{i}(x) > 0 \}, \\
    \IndexSet_{0-}(x) &\define \{i \mid H_{i}(x) = 0, G_{i}(x) < 0 \}, \\
    \IndexSet_{0}(x) &\define \IndexSet_{00}(x) \cup \IndexSet_{0+}(x) \cup \IndexSet_{0-}(x),
    \end{aligned} & \quad
                    \begin{aligned}[t]
                      \IndexSet_{+0}(x) &\define \{i \mid H_{i}(x) > 0, G_{i}(x) = 0\}, \\
                      \IndexSet_{++}(x) &\define \{i \mid H_{i}(x) > 0, G_{i}(x) > 0\}, \\
                      \IndexSet_{+}(x) & \define \IndexSet_{+0}(x) \cup \IndexSet_{++}(x),
                    \end{aligned}
  \end{aligned}
\end{equation*}
where we call the indices in $\IndexSet_{00}(x)$ and respective constraints \emph{bi-active} (see Figure \ref{fig:MPVC_vertical}). 
\begin{figure}[ht]
  \centering
  \includegraphics[width=.3\textwidth]{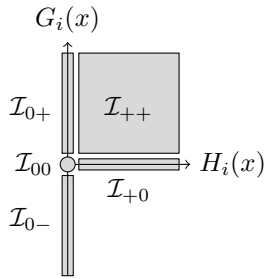}
  \caption{The feasible region of a controlling / vanishing constraint in vertical form.}
  \label{fig:MPVC_vertical}
\end{figure}
The reason for the difficulty of solving programs with vanishing
constraints is the loss of strong constraint qualifications. For
instance, the well-known \emph{Linear Independence Constraint
  Qualification (LICQ)}~\cite{numerical_optimization} as well as the
slightly weaker \emph{Mangasarian-Fromovitz Constraint Qualification}
(MFCQ)~\cite{mfcq} are generally not satisfied at bi-active points of
MPVCs (see~\cite{mpvc} for an in-depth explanation). Consequently, as
any nonlinear solver approaches such points, the determination of
suitable Lagrange multipliers satisfying the KKT conditions becomes hard
or even impossible.

MPVCs are closely related to so-called \emph{Mathematical Programs
  with Equilibrium Constraints}, called MPECs for short~\cite{mpec},
which are identical to MPVCs except for the fact that the products
$G_{i}H_{i}$ must be zero rather than non-negative. Indeed, MPVCs can
be transformed to MPECs by introducing additional slack
variables~\cite[Lemma 1]{mpvc}.  The problem class consisting of MPVCs
and MPECs is therefore also often referred to as
\emph{Mathematical Programs with Complementarity Constraints} (MPCCs).

From a theoretical view, it is necessary to derive a concept of
stationarity in order to obtain practically useful first-order
optimality conditions for MPCCs.  Consequently, several different
stationarity concepts valid under varying problem-specific constraint
qualifications particularly for MPECs have been introduced, such as
B-stationarity~\cite{abadie_type_constraint_qualification}, W- and
S-stationarity~\cite{mpec_weak_stat}.
These stationarity concepts have great effect on the practical
tractability, where weaker stationarity concepts allow for the design
of reliably convergent algorithms (see e.g.~\cite{benko_sqp}).  In
certain applications it is however crucial to obtain points satisfying
\emph{strong stationarity} conditions (which we will define later in
Theorem~\ref{thm:mpvc_kkt} in
Section~\ref{sec:constraints_qualifications}), rendering such
approaches useless. One method converging to points satisfying strong
stationarity conditions is an SQP method presented in
\cite{kirches_sqp}.

Another approach to solving generic MPCCs is to artificially introduce
regularity by adding a regularization term to the vanishing
constraints (see
e.g.~\cite{diss_hoheisel,Hoheisel22,relaxation_method,solving_mpccs} for a summary). For instance, for MPEC
problems a parameter $\sigma > 0$ is chosen, and the equilibrium
constraints $G_{i}H_{i} = 0$ are replaced by
$\Phi(G_{i}, H_{i}, \sigma) = 0$, where
$\Phi : \Real \times \Real \times \Real \to \Real$ tends to
$G_{i}H_{i}$ for $\sigma \to 0$ while satisfying LICQ or MFCQ for
$\sigma > 0$. The optimization is then carried out by solving a
sequence of NLPs parameterized in $\sigma$, which is driven to
zero. At each step, the optimal solution of one NLP is used to
initialize the next one. A popular regularization function is due
to~\cite{scholtes}, where
$\Phi(G_{i}, H_{i}, \sigma) = G_{i}H_{i} - \sigma$. Regardless of the
function however, the relaxation parameter must be carefully selected
or adjusted through an iterative process and there remains a risk that 
exact feasibility is achieved only in the limit as the 
relaxation parameter tends to zero.

Regardless of the solution approach, it is however often convenient
to reformulate MPVCs in order to separate the nonlinear
constraints from the combinatorial structure of
the feasible set imposed by the vanishing constraints (VC) and
controlling constraints (CC). Specifically, by introducing additional
(otherwise unconstrained) slack variables, we can ensure that the
vanishing and controlling constraints $G_{i}$ and $H_{i}$ consist of
individual variables, \ie $G_{i}(x) = s_{\idxVC{i}}$ and
$H_{i}(x) = s_{\idxCC{i}}$ for $\idxVC{i}, \idxCC{i} = 1, \ldots,
n$. We combine the slack variables in the vector $s\in\Real^{2l}$.  We
may require that the slack variable bounds are non-trivial in the
sense that the lower bounds $s_\mathrm{l}\in\Real^{2l}$ are zero at
all entries to the indices of $\idxCC{i}$ and not positive for the
entries $\idxVC{i}$ as well the upper bounds
$s_\mathrm{u}\in\Real^{2l}$ are not negative respectively at all
entries.  We define
$C^\mathrm{s} \define \{s \in \Real^{2l} \mid s_\mathrm{l} \leq s \leq
s_\mathrm{u} \}$ as the convex set of the slack variables.
Such a structure of~\eqref{eq:mpvc} is called its \emph{vertical form}.
\begin{equation}
  \label{eq:mpvcs}
  \tag{MPVCS}
  \begin{aligned}
    \min_{x \in C,s \in C^\mathrm{s}} \quad & f(x) \\
    \st \quad & g(x) = 0, && \\
    & H_{i}(x)-s_{\idxCC{i}}= 0, && \text{for } i = 1, \ldots , l, \\
    & G_{i}(x)-s_{\idxVC{i}}= 0, && \text{for } i = 1, \ldots , l, \\
    & s_{\idxCC{i}}\geq0, && \text{for } i = 1, \ldots , l, \\
    & s_{\idxVC{i}}s_{\idxCC{i}}\geq0, && \text{for } i = 1, \ldots , l. \\
  \end{aligned}
\end{equation}
Throughout this paper, we will follow an approach similar
to~\cite{kirches_sqp}, consisting of incorporating constraints and
their bounds into local subproblems in different ways, depending on
the current index set. However, in our approach, we reduce the
complexity by focusing on the biactive points and following actions
along the solution path. Specifically, after discussing the
theoretical foundations of the problem in
Section~\ref{sec:constraints_qualifications}, we will propose a
practical algorithm in Section~\ref{sec:algorithm}, which we base on
the concept of gradient flows. After discussing
implementational details in Sections~\ref{sec:implementation}
and~\ref{sec:nonlinear}, we go on
to present numerical results on truss topology design problems in
Section~\ref{sec:num_ex}.

\Paragraph{Notation}

As in \cite{sequential_homotopy}, we let $\primalProd{x}{x'}$ be a
scalar product inducing the norm $\primalNorm{\cdot}$ for two points
$x, x' \in X$ in a real finite dimensional Hilbert space $X$. Likewise we get a scalar
product real finite dimensional Hilbert spaces $Y,\,\tilde{S}$.

\section{Constraint Qualifications}\label{sec:constraints_qualifications}

We begin by introducing a concept of stationarity in order to obtain
practically useful first-order optimality conditions for
MPCCs. Specifically, we rely on the MPVC \emph{strong stationarity}
concept introduced in~\cite{mpvc_stat_cons}. To this end, we let
\begin{equation}
  \label{eq:nmpv_nlp}
  \tag{MPVC-NLP}
  \min \: \: \tilde{f}(\tilde{x}) \:\: \st \: \: \tilde{g}(\tilde{x}) = 0, \:\: \tilde{h}(\tilde{x}) \geq 0
\end{equation}
be \eqref{eq:mpvc} as a standard NLP, where
$\tilde{f} : \Real^{\tilde{n}} \to \Real$,
$\tilde{g} : \Real^{\tilde{n}} \to \Real^{\tilde{m}}$, and
$\tilde{h} : \Real^{\tilde{n}} \to \Real^{\tilde{p}}$ and
$\tilde{M} \define \{ x \mid \tilde{g}(\tilde{x}) \leq 0, \tilde{h}(\tilde{x}) = 0 \}$ be its set of feasible points.  The
\emph{linearized cone} \cite{mpvc} of the reformulation at a point
$\tilde{x} \in \tilde{M}$ is given by
\begin{equation}\label{eq:linearized_cone}
  \begin{aligned}
      L(\tilde{x}) \define \{d \in \Real^{\tilde{n}} \mid & 
    \nabla \tilde{g}_{i}(\tilde{x}) d = 0 \text{ for } j = 1, \ldots, \tilde{p}, \text{ and}\\
    & \nabla \tilde{h}_{i}(\tilde{x}) d \leq 0 \text{ for } j\;\st\; \tilde{h}_{i}(\tilde{x}) = 0 \}\\
    = \{d \in \Real^{\tilde{n}} \mid &  \nabla g_j (\tilde{x})d = 0 \text{ for }j=1,\ldots,m,\\
    & \nabla H_i(\tilde{x}) d = 0 \text{ for } i \in \IndexSet_{0-}(\tilde{x}),\\
    & \nabla H_i(\tilde{x}) d \leq 0 \text{ for } i \in \IndexSet_{00}(\tilde{x})\cup\IndexSet_{0+}(\tilde{x}),\\
    & \nabla G_i(\tilde{x}) d \leq 0 \text{ for } i \in \IndexSet_{+0}(\tilde{x}) \}\\
  \end{aligned}
\end{equation}
Furthermore, the \emph{tangent cone} of a set $M \subseteq \Real^{n}$ at $x \in M$ is given by
\begin{equation*}
  \begin{aligned}
      T(M ,x) \define \{d \in \Real^{n} \mid &
  \text{ there exist sequences} (x_{k}) \in M, (\lambda_{k}) \subset \Real_{\geq 0}
\\
    &    \text{ with }  x_k \to x
  \text{ and } \lambda_{k} (x_{k} - x) \to d \text{ for } k \to \infty
  \}
  \end{aligned}
\end{equation*}
and the \emph{polar cone} of a cone $M \subseteq \Real^{\tilde{n}}$ is given by
\begin{equation*}
  M^{-} \define \{d \in \Real^{\tilde{n}} \mid \langle d, x \rangle \leq 0 \text{ for all } x \in M\},
\end{equation*}
where $\langle \cdot, \cdot \rangle$ denotes the standard scalar product.
A point $\tilde{x} \in \tilde{M}$ is said to satisfy
the \emph{Guignard Constraint Qualification} (GCQ)~\cite{guignard_stat_cons} iff $T^{-}(\tilde{M}, \tilde{x}) = L^{-}(\tilde{x})$. This more geometric condition is significantly weaker than both LICQ and MFCQ and has a good chance to hold in practice (see~\cite{on_the_abadie} for a discussion). Based on this constraint qualification the following stationarity condition holds:
\begin{theorem}[Theorem 2.5 in~\cite{mpvc_stat_cons}]
  \label{thm:mpvc_kkt}
  Let $x^{*}$ be a local minimum of~\eqref{eq:mpvc} such that GCQ holds at $x^{*}$. Then there exist Lagrange multipliers $y^{*} \in \Real^{m}$,
  $\eta_{H}^{*}, \eta_{G}^{*} \in \Real^{l}$
  such that
  \begin{equation*}
    - \nabla_{x} \Lag_{\MPVC}(x^{*}, y^{*}, \eta_{G}^{*}, \eta_{H}^{*}) \in T^{-}(C, x^{*}),
  \end{equation*}
  where
  \begin{equation*}
    \Lag_{\MPVC}(x, y, \eta) \define f(x) + \langle y, g(x) \rangle + \sum_{i = 1}^{l} (\eta_{H})_i H_{i}(x) + \sum_{i = 1}^{l} (\eta_{G})_i G_{i}(x)
  \end{equation*}
  and
  \begin{equation*}
    \begin{aligned}
      (\eta_{H})_i
      \begin{cases}
        = 0 & \text{ if } i \in \IndexSet_{+} \\
        \leq 0 & \text{ if } i \in \IndexSet_{00} \cup \IndexSet_{0+} \\
        \text{free} & \text{ if } i \in \IndexSet_{0-}
      \end{cases}
      \quad
      \text{ and }
      \quad
      (\eta_{G})_i
      \begin{cases}
        = 0 & \text{ if } i \in \IndexSet_{0} \cup \IndexSet_{++} \\
        \leq 0 & \text{ if } i \in \IndexSet_{+0}.
      \end{cases}
    \end{aligned}
  \end{equation*}
\end{theorem}

\section{A Practical Algorithm}
\label{sec:algorithm}

\Paragraph{Gradient Flows}

In order to solve MPVCs we examine primal-dual gradient /
anti-gradient flows~\cite{sequential_homotopy}. We will therefore
begin by giving an introduction regarding gradient flows and their use
in nonlinear programming. For a nonlinear program
\begin{equation}
  \label{eq:nlp}
  \tag{NLP}
  \min \:\: f(x)
  \: \: \st \:\: g(x) = 0, \text{ and } x \in C \\
\end{equation}
we let $\Lag(x, y) \define f(x) + \dualProd{y}{g(x)}$ be its Lagrangian function,
and $\Lag^{\rho}(x, y) \define \Lag(x, y) + \tfrac{\rho}{2} \dualNorm{g(x)}^{2}$
its \emph{augmented Lagrangian} for a parameter $\rho \geq 0$. Based on the
augmented Lagrangian the \emph{gradient / anti-gradient flow} emanating from the initial values $x_{0} \in C$ and $y_{0} \in \Real^{m}$
is defined by the following initial value problem for projected differential equations
\begin{equation*}
  \begin{aligned}
    \dot{x}(t) &= P_{T(C, x(t))} \left( -\nabla_{x} \Lag^{\rho} (x(t), y(t)) \right) \text{, and }\:
    \dot{y}(t) = \nabla_{y} \Lag^{\rho}(x(t), y(t)), \text{ with }\\
    x(0) &= x_0 \text{, and } y(0) = y_0.
  \end{aligned}
\end{equation*}
where $P_{M}(x)$ denotes the $\primalProd{\cdot}{\cdot}$-orthogonal
projection of $x$ onto the convex set $M$. We would like to point out
that $P_{M}$ is in general non-differentiable. The
projection operator is however \emph{semismooth}~\cite{semismooth}
with a set-valued derivative in the sense of Clark, enabling the
solution of the equations. Furthermore, it
was shown in~\cite{sequential_homotopy} that the equilibrium points of
this flow, \ie those with $\dot{x}(t) = 0$ and $\dot{y}(t) = 0$ are
precisely the stationary points of~\eqref{eq:nlp}.  Note that the flow
is smooth as long as $T(C, x(t))$ stays constant and non-smooth but
continuous during changes of the tangent cone, which occur during
changes in the active set of the variable bounds constituting $C$.

To find a stationary point using gradient flows, we solve the
initial value problem until such a time where $\dot{x}(t)$ and
$\dot{y}(t)$ have been reduced to an acceptable level, or some other
convergence criterion is satisfied.  In principle, any integration
method, such as for example a projected forward Euler method can be
applied to gradient flow problems.  However, since gradient /
anti-gradient flows tend to be stiff especially for $\rho \gg 1$, an implicit method is preferable
in practice.  We therefore choose to compute projected backward Euler
steps for a step size of $\Delta t > 0$ and a current point
$\hat{x} \in C$, $\hat{y} \in \Real^{m}$ by solving the equations
\begin{equation*}
  x - P_{C} \left( \hat{x} - \Delta t \nabla_x \Lag^{\rho}(x, y) \right) = 0
  \:  \text{, and }\:
  y - \hat{y} - \Delta t \nabla_{y} \Lag^{\rho}(x, y) = 0.
\end{equation*}
The solution of these equations involving the semismooth operator
$P_{C}$ can be carried out using a semismooth Newton (SSN)
method~\cite{ssn}.  Unfortunately, the use of SSN methods
poses its own challenges, in particular due to the fact that these
methods are very hard to globalize and suffer from small radii of
local convergence.  We therefore opt to interpret the projected
backward Euler system as a necessary optimality condition of the
following regularized variant of~\eqref{eq:nlp},
\begin{equation}
  \tag{Sub-NLP}
  \label{eq:sub_nlp}
  \begin{aligned}
    \min_{\mathclap{x \in C, w \in \Real^{m}}} & \quad \quad f^{\rho}(x) + \lambda \left[ \tfrac{1}{2} \primalNorm{x - \hat{x}}^{2} + \tfrac{1}{2} \dualNorm{w - \hat{y}}^{2} \right] \\
    \st & \quad \quad c(x) + \lambda w = 0
  \end{aligned}
\end{equation}
where $f^{\rho}(x) \define f(x) + \tfrac{\rho}{2} \dualNorm{c(x)}^{2}$ and $\lambda\define 1/\Delta t$ (see \cite{sequential_homotopy}).
Thus, to compute the step, we can solve this problem with an off-the-shelf NLP solver. The regularization with $\lambda$ ensures 
that the problem always satisfies LICQ for all $\lambda>0$ (under standard assumptions on $f$, $c$, 
and $C$). We would like to point out that the variable $w$ can in principle be eliminated from the formulation entirely, 
yielding a box-constrained problem instead. While this approach seems sensible, this reformulation does however greatly 
increase the nonlinearity of the subproblem. As a result, a significantly larger value of $\lambda$ is required for 
regularization, impeding the convergence of the gradient flow itself.

\Paragraph{Branch Decomposition}

In the following, we examine the combinatorial structure
of~\eqref{eq:mpvcs} in vertical form. Clearly, the feasible region with
respect to the slack variable constraints $s_{\idxVC{i}} \geq 0$ and
$s_{\idxVC{i}}s_{\idxCC{i}} \geq 0$ can be decomposed into two \emph{branches}. 
We define for every pair $(s_{\idxVC{i}},s_{\idxCC{i}})$ of slack variables for $i=1,\dots,l$, the \emph{branches} 
\begin{equation}\label{eq:convex_pieces}
  \begin{aligned}
    C_{+}(i) &\define \{ s_{\idxCC{i}} \geq 0 \text{ and } s_{\idxVC{i}} \geq 0 \}, \text{ and} \\
    C_{0}(i) &\define \{ s_{\idxCC{i}} = 0 \text{ and } s_{\idxVC{i}} \leq 0 \},
  \end{aligned}
\end{equation}
where the intersection $C_{+}(i) \cap C_{0}(i)$ consists of the biactive points, \ie 
$i \in \IndexSet_{00}(x)$ for $H_i(x)$ and $G_i(x)$. We call $C_{+}(i)$ the \emph{upper branch} and $C_{0}(i)$ the \emph{lower branch}. 
It is clear that we can always modify an instance to ensure that these conditions are met.

The feasible region for the slack variables of~\eqref{eq:mpvcs} can be decomposed into a number of 
convex pieces based on the individual branches. Specifically, for
a \emph{signature} $\signature \in \{0, 1\}^{l}$, we can set
\begin{equation*}
C^\mathrm{b}(\signature) \define C^\mathrm{s} \cap \left( \bigtimes_{i = 1}^{l}
\begin{cases}
  C_{+}(i) & \text{ if } \signature_{i} = 1 \\
  C_{0}(i) & \text{ if } \signature_{i} = 0,
\end{cases}
  \right)
\end{equation*}
to be the convex piece associated with $\signature$. For this piece, we can
investigate the optimization problem
\begin{equation}
  \label{eq:branch_nlp}
  \tag{NLP$\sigma$}
    \begin{aligned}
    \min_{x \in C,s \in C^\mathrm{b}(\signature)} \quad & f(x) \\
    \st \quad & g(x) = 0  &&\\
    & H_{i}(x)-s_{\idxCC{i}}= 0 && \text{for } i = 1, \ldots , l \\
    & G_{i}(x)-s_{\idxVC{i}}= 0 && \text{for } i = 1, \ldots , l 
  \end{aligned}
\end{equation}
over the convex domain $C$ and $C^\mathrm{b}(\signature)$. This optimization problem no longer contains
any vanishing constraints, and standard arguments show that local
minima must satisfy the stationarity criterion of
$-\Lag(x^{*},s^{*},y^{*}) \in T^{-}(C \times C^\mathrm{b}(\signature), (x^{*},s^{*}))$. As a result, we could
try to solve~\eqref{eq:mpvcs} by enumerating over its convex pieces of slack variables
while solving a nonlinear program for each individual one and checking
for stationarity. Of course, the number of $2^{l}$ pieces makes such
an approach intractable in practice.

It is however worth discussing the relationship between the
stationarity criterion for~\eqref{eq:mpvc} from Theorem~\ref{thm:mpvc_kkt} 
and its counterpart for~\eqref{eq:branch_nlp}. Since problem~\eqref{eq:mpvcs} is a reformulation of~\eqref{eq:mpvc},
the Lagrange multipliers of the slack variables in a solution of~\eqref{eq:mpvcs} coincide with the Lagrange
multipliers of vanishing and controlling constraints in~\eqref{eq:mpvc}. We therefore only need to discuss
the relationship between the slack variables of~\eqref{eq:mpvcs} and~\eqref{eq:branch_nlp}.
For short we denote  $g^{x,*} \define -\nabla_{x} \Lag(x^{*},s^{*}, y^{*}) =-\nabla_{x} \Lag_{\MPVCS}(x^{*},s^{*}, y^{*})$,
$g^{s,*} \define -\nabla_{s} \Lag(x^{*},s^{*}, y^{*})$, and $\overline{g}^{s,*} =-\nabla_{s} \Lag_{\MPVCS}(x^{*},s^{*}, y^{*})$. 
The equality of the Lagrange gradients $g^{x,*}$ holds by definition of~\eqref{eq:branch_nlp}.
Since in both problem formulations~\eqref{eq:mpvcs} and~\eqref{eq:branch_nlp} the convex set $C$ does not change we only need to
examine the $j$-th component of the respective Lagrangian gradients $g^{s,*}$ for active constraints and given signature $\signature$. The
following cases arise:
\begin{itemize}
\item If $j$ is part of the slack variables corresponding to $H_i$ and $G_i$ for an $i\in1,\dots,l$, \ie $j=\idxCC{i}$ or
  $j = \idxVC{i}$, such that $C_{+}(i)\cap C^\mathrm{s}\subseteq C^\mathrm{b}(\signature)$. 
  Consequently, $i$ must be in one of the sets $\IndexSet_{00}$, $\IndexSet_{0+}$,
  $\IndexSet_{+0}$, or $\IndexSet_{++}$.  
  For $j = \idxCC{i}$ and $i \in \IndexSet_{+0}$ the case is symmetric and we have $(\eta_{H})_{i}=0$. 
  For $j = \idxVC{i}$ and $i \in \IndexSet_{+0}$, we can observe that 
  $T^{-}(C\times C^\mathrm{b}(\signature), (x^{*},s^{*}))$ requires the $j$-th
  component of $g^{s,*}$ to be non-positive, whereas that component has
  to be zero in $T^{-}(C\times C^\mathrm{s}, (x^{*},s^{*}))$ due to the non-trivially chosen
  variables bounds. Since the otherwise unconstrained
  $(\eta_{G})_{i}$ can be set to any non-positive value, the residuals
  can be made to coincide. 
  If $j = \idxVC{i}$ corresponds to the VC $G_i$ and 
  $i \in \IndexSet_{0+}$, the corresponding entry of $(\eta_{G})_{i}$ 
  must be zero and $(\eta_{H})_{i}$ greater equal zero and the stationarity residuals again coincide. 
  Thus, the only remaining case occurs for
  $i \in \IndexSet_{00}$: Here, $T^{-}(C\times C^\mathrm{b}(\signature), (x^{*},s^{*}))$ requires both
  $g^{s,*}_{\idxCC{i}}$ and $g^{s,*}_{\idxVC{i}}$ to be non-positive, whereas
  $T^{-}(C\times C^\mathrm{s}, (x^{*},s^{*}))$ requires $\overline{g}^{s,*}_{\idxVC{i}}$ to be zero, which is in general  
  not possible. Thus, if $g^{s,*}_{j} < 0$ and $j = \idxVC{i}$,
  the stationary point for~\eqref{eq:branch_nlp} in the current convex
  piece does not coincide with a stationary point for~\eqref{eq:mpvcs}.
\item If $j$ is part of the slack variables corresponding to a vanishing constraint $\idxVC{i}$ or controlling constraint $\idxCC{i}$ such that
  $C_{0}(i)\cap C^\mathrm{s} \subseteq C(\signature)$. The index set of interest is $i \in \IndexSet_{00}$:
  As before we need to have $g^{s,*}_{\idxVC{i}} =0$, since $T^{-}(C\times C^\mathrm{s}, (x^{*},s^{*}))$ requires $\overline{g}^{s,*}_{\idxVC{i}}=0$.
  If $g^{s,*}_{j} > 0$ and $j = \idxVC{i}$ the stationary point for~\eqref{eq:branch_nlp} in the current convex
  piece is therefore not stationary for~\eqref{eq:mpvcs}. 
\end{itemize}
Based on these observations it is easy to see that the stationarity
criterion for~\eqref{eq:mpvc} is stricter than that for individual
branches in the sense that a stationary point for~\eqref{eq:mpvc} is
stationary for~\eqref{eq:branch_nlp} for all fitting choices of $\signature$, i.e., $s^*\in C^\mathrm{s}\cap C^\mathrm{b}(\signature)$, whereas
the converse does not hold in general. Furthermore, the sign of
$g^{s,*}$ in the direction of the $\idxVC{i}$-entries indicates the
non-stationarity of the primal-dual variables $(x^{*},s^{*},y^{*})$.

\Paragraph{Piecewise Gradient Flows}

To solve~\eqref{eq:mpvc}, we begin at an initial guess
$x_{0} \in C$ and $s_{0} \in C^\mathrm{b}(\signature)$ for some signature $\signature$. This
point does not need to satisfy the nonlinear constraint $g$. Based on
the set $C\times C^\mathrm{b}(\signature)$ and an initial estimate $y_{0} \in \Real^{m}$, we can
consider the gradient flow defined by the equations
\begin{equation*}
  \begin{aligned}
  (\dot{x}(t),\dot{s}(t)) &= P_{T(C\times C^\mathrm{b}(\signature), (x(t),s(t)))} \left( -\nabla_{x,s} \Lag^{\rho} (x(t), s(t), y(t)) \right) , \quad\\
    \dot{y}(t) &= \nabla_{y} \Lag^{\rho}(x(t), s(t), y(t))\text{ with }\\
    x(0) &= x_0,~s(0)=s_0 \text{, and } y(0) = y_0.\\
  \end{aligned}
\end{equation*}
This flow is of course designed to solve
problem~\eqref{eq:branch_nlp}.  However, as discussed earlier, even a
stationary point for $C\times C^\mathrm{b}(\signature)$ is not generally going to be a solution
of~\eqref{eq:mpvcs}. It is therefore necessary for the flow to leave
one convex piece and switch to another while preserving continuity.
Of course a continuous switch between the different branches can only
be conducted at times $t > 0$ where the slack variables $s(t)$ of the
flow becomes bi-active for some $i$. At such a time, we are free to switch the
signature by flipping one or more branches for
$i \in \IndexSet_{00}(x(t))$ while maintaining continuity (but not
smoothness) of the resulting flow, which we combine with a detection
of the stationarity of~\eqref{eq:mpvcs}, to prevent the
continuation of the flow into any other branch. Thus, for a given
primal-dual $(x,s,y)$ with signature $\signature$ we start by determining
whether the pair is stationary according to the conditions of
Theorem~\ref{thm:mpvc_kkt}. To this end, the multipliers $\eta_{H}$
and $\eta_{G}$ are chosen in such a way as to minimize the absolute
value of the corresponding components of
$g_{\MPVCS} \define - \nabla_{x,s} \Lag_{\MPVCS}(x,s,y)$ subject to the
sign constraints of the multipliers. The stationarity residual is then
given by the distance of $g_{\MPVCS} - P_{T^{-}(C\times C^\mathrm{s}, (x^{*},s^{*}))}(g_{\MPVCS})$
with respect to an appropriate norm. Once this residual is below
some threshold $\epsilon$, the solution is (approximately)
stationary. If stationarity is not achieved, switching candidates are
identified using $-\nabla_{x,s} \Lag^{\rho}(x,s,y)$, yielding a new
signature $\signature'$ and updated branches:

\begin{itemize}
\item If $i \in \IndexSet_{00}(x)$ and $\signature_{i}=1$ we switch to the lower branch $C_0(i)$ if
  \begin{equation*}
    (-\nabla_{s} \Lag^{\rho}(x,s,y))_{\idxCC{i}}\leq 0 \quad\textrm{and}\quad
    (-\nabla_{s} \Lag^{\rho}(x,s,y))_{\idxVC{i}} < 0,
  \end{equation*} which means the slack variable $s_{\idxVC{i}}$ will vanish.
  We update the signature by setting $\signature'_{i} = 0$.
\item If $i \in \IndexSet_{00}(x)$ and $\signature_{i} = 0$ we switch to the upper branch $C_+(i)$ if 
  \begin{equation*}
    (-\nabla_{s} \Lag^{\rho}(x,s,y))_{\idxVC{i}} > 0, 
  \end{equation*} or \begin{equation*}
    (-\nabla_{s} \Lag^{\rho}(x,s,y))_{\idxVC{i}} = 0 \quad \text{and}\quad (-\nabla_{s} \Lag^{\rho}(x,s,y))_{\idxCC{i}} > 0,
  \end{equation*} which means the slack variable $s_{\idxVC{i}}$ is free in the next piece of the gradient flow.
  We update the signature by setting $\signature'_{i} = 1$.
\item
  The bounds and signature values of all other variables are kept constant.
\end{itemize}
Once  $\signature'$ has been chosen, we continue with the next piece
of our piecewise gradient flow at the point $(x,s, y)$ (with
$(x,s) \in C\times C^\mathrm{b}(\signature')$) and integrate along the signature-modified initial value problem by setting \begin{equation*}
\begin{aligned}
  (\dot{x}(t),\dot{s}(t)) &= P_{T(C\times C^\mathrm{b}(\signature'), (x(t),s(t)))} \left( -\nabla_{x,s} \Lag^{\rho} (x(t), s(t), y(t)) \right) , \quad\\
    \dot{y}(t) &= \nabla_{y} \Lag^{\rho}(x(t), s(t), y(t))\text{ with }\\
    x(0) &= x,~s(0)=s \text{, and } y(0) = y.\\
  \end{aligned}
\end{equation*}
Note, $\dot{y}(t)$ is independent of the choice of signatures and remains unchanged.

\section{Implementational Aspects}\label{sec:implementation}

In the following, we discuss some aspects of the implementation required for the solution of particularly challenging instances.

\Paragraph{Algorithmic Parameters} The choice of the algorithmic
parameters $\lambda$ and $\rho$ has a large impact on the
convergence speed of our method. The parameter $\lambda$ (the inverse
of the step size $\Delta t$) controls the progress of the gradient
/ anti-gradient flow towards stationary points. Thus, we would like
to decrease $\lambda$ as far as possible. However, if the risk of ill-conditioned \eqref{eq:sub_nlp} increases. 
To prevent divergence of subsystem solvers,
we propose a simple heuristic to update $\lambda$ after each iteration:
Based on an initial value of $\lambda^{(0)} = \lambda_{\init}$, we set
\begin{equation*}
  \lambda^{(k + 1)} \define
  \begin{cases}
    \lambda^{-1}_{\inc} \cdot \lambda^{(k)} & \text{ if current \eqref{eq:sub_nlp} was solved successfully} \\
    \lambda_{\inc} \cdot \lambda^{(k)} & \text{ otherwise,} \\
  \end{cases}
\end{equation*}
where we set $\lambda_{\init} = \num{0.1}$ and
$\lambda_{\inc} = \num{2}$.

In contrast to the value of $\lambda$, the method is less dependent on the choice of the parameter $\rho$. In many cases, a small constant value of
$\rho$ is sufficient. As an alternative, we propose a simple rule
to update $\rho$ which works by ensuring that, whenever possible,
the value $\tfrac{1}{2}\|c(x(t))\|^2$ is non-increasing along the trajectory.
To ensure this property, we note that
\begin{equation*}
  \frac{\mathrm{d}}{\mathrm{d}t} \frac{1}{2}\|c(x(t))\|^2=
  - \langle d, P' (\nabla f(x(t)) + y(t)^{T} J_{c}(x(t))) \rangle
  - \rho \langle d, P'd \rangle
\end{equation*}
where $d \define c(x(t))^{T} J_{c}(x(t))$ and $P' \in \Real^{n \times n}$ is an element of the set-valued derivative of the projection operator $P$. Consequently, as long
as $\langle d, P'd \rangle < 0$ we can increase $\rho$ sufficiently in order to ensure that $\tfrac{1}{2}\|c(x(t))\|^2$ does not increase. In practice, our approach is more conservative and we increase $\rho$ at most by a factor of $\rho_{\inc} = 10$ per iteration.

\Paragraph{Scaling}
The aspect of scaling plays an important role during the implementation
of nonlinear programming codes~\cite[Sec. 3.8]{ipopt}, as it can
dramatically speed up the convergence on badly scaled real-world
instances. To include scaling in our algorithm, we employ the scalar
products $\primalProd{\cdot}{\cdot}$ and $\dualProd{\cdot}{\cdot}$
whose induced norms appear in \eqref{eq:sub_nlp}. Since the
stationarity conditions are independent of the underlying norm of the
space, we are free to employ these scalar products to improve
convergence. Specifically, rather than the standard Euclidean scalar
product, it may be convenient to use scalar products induced by
scaling matrices $D_{x}$, $D_{s}$ and $D_{y}$ instead and solve \begin{equation}\label{eq:scaled_nlp}
  \begin{aligned}
    \min_{\mathclap{x \in C, s\in C^\mathrm{b}(\signature), w \in \Real^{m}}} \quad& \quad \quad f^{\rho}(x,s) + 
    \lambda \big( \tfrac{1}{2} \matrixScalarprod{x - \hat{x}}{D_x} \\
    &\quad\quad\quad+ \tfrac{1}{2} \matrixScalarprod{s - \hat{s}}{D_s} + \tfrac{1}{2} \matrixScalarprod{w - \hat{y}}{D_y}\big) \\
    \st \quad& \quad \quad c(x,s) + \lambda w = 0
  \end{aligned}
\end{equation}
Furthermore, it can be advantageous to scale the variables itself.

\section{Vanishing Constraints}\label{sec:nonlinear}

Let us consider an instance of~\eqref{eq:mpvcs} in
vertical form to use this additional structure to realize an efficient implementation.

\Paragraph{Initializing Branches} To find a suitable branch and feasible starting values for the slack variables we evaluate
at a given primal initial guess the controlling and vanishing constraints: \begin{itemize}
  \item If $H_i(x)>0$ we set $s_{\idxCC{i}}= H_i(x),~ s_{\idxVC{i}}=\max\{0,G_i(x)\}$ and we choose the branch $C_+(i)$.
  \item If $H_i(x)=0$ we set $s_{\idxCC{i}} = 0,~ s_{\idxVC{i}}=G_i(x)$. Depending on the value $s_{\idxVC{i}}$ 
  we choose $C_+(i)$ if $s_{\idxVC{i}}>0$ and $C_0(i)$ otherwise. 
  \item If $H_i(x)<0$ we set $s_{\idxCC{i}} = 0$ and choose the lower branch $C_0(i)$ and set $s_{\idxVC{i}}=\min\{0,G_i(x)\}$.  
\end{itemize}

\Paragraph{Reduced Subsequent NLPs}

Note, that the lower branch \( C_0 \) from \eqref{eq:convex_pieces} corresponds to the set of vanishing constraints
indexed by \( \IndexSet_{0-} \). In this case, the controlling constraints in \( \IndexSet_{0-} \) are treated as equality constraints. 
This observation aligns with the structure of the linearized cone presented in \eqref{eq:linearized_cone}.
More precisely, the vanishing constraints \( G_i \), respectively the slack variables $s_{\idxVC{i}}$, 
for \( i \in \IndexSet_{0-} \) have vanished from the linearized cone, as they do not contribute to feasibility 
or optimality conditions at a local minimum. This is also reflected in the zero Lagrange multipliers for 
the vanishing constraints in \(\IndexSet_{0-} \) at a strong stationary point. Consequently, these vanished 
constraints, i.e.,  the constraints in the lower branch $C_0$ can be omitted in the formulation of the subsequent NLP.
A similar idea can be found in \cite{kirches_sqp}. Let $\signature$ be a signature for the branch decomposition. 
We define the reduced slack variables 
\begin{equation*}
  \tilde{s} \define (s_{\idxCC{1}},\ldots,s_{\idxCC{i}}, (s_{\idxVC{i}},\;\text{for}\;\signature_i = 1))^T,
\end{equation*}
with 
\begin{equation*}
  C^\mathrm{r}(\signature) \define C^\mathrm{s} \cap \left(\bigtimes_{i = 1}^{l}
\begin{cases}
  C_{+}(i) & \text{ if } \signature_{i} = 1 \\
  \{s_{\idxCC{i}} = 0\} & \text{ if } \signature_{i} = 0.
\end{cases} \right)
\end{equation*}
and respective real Hilbert space $\tilde{S}$ with its induced norm. The reduced constraints are
\begin{equation*}
  \tilde{c}(x,\tilde{s}) \define
  \begin{pmatrix*}[l]
      \multicolumn{2}{c}{g(x)} \\
      H_i(x)-s_{\idxCC{i}} & \text{for } i=1,\ldots,l\\
      G_i(x)-s_{\idxVC{i}} & \text{for } \signature_i=1
    \end{pmatrix*} \in \Real^r
\end{equation*}
with $r = m+l+\#\{i\,\mid\,\signature_i = 1\}$. We consider the following reduced problem:
\begin{equation}\label{eq: Red-NLP}
  \tag{Red-NLP}
  \begin{aligned}
      \quad\quad \quad\min_{\mathclap{x \in C, \tilde{s}\in C^\mathrm{r}(\signature), \tilde{w} \in \Real^r}} &  \quad\quad\quad \tilde{f}^{\rho}(x,\tilde{s}) 
    + \lambda \left[ \tfrac{1}{2} \primalNorm{x - \hat{x}}^{2} 
        + \tfrac{1}{2} \|\tilde{s} - \tilde{\hat{s}}\|_{\tilde{S}} 
        + \tfrac{1}{2} \dualNorm{\tilde{w}- \tilde{\hat{y}}}^2 
    \right] \\
    \text{s.t.} & \quad \tilde{c}(\tilde{x}) + \lambda \tilde{w} = 0,
  \end{aligned}
\end{equation}
where \( \tilde{w} \) is a vector of the same dimension as \( \tilde{c} \), 
\( \tilde{f}^{\rho}(x,\tilde{s}) \) denotes the reduced augmented objective, 
and \( \tilde{\hat{x}} \), \(\tilde{\hat{s}}\), and \( \tilde{\hat{y}} \) are 
the corresponding reduced solutions from the previous iteration.

Following the solution of \eqref{eq: Red-NLP}, the slack variables \( s_{\idxVC{i}} \) for \( \signature_i = 0 \) 
are updated such that they satisfy  $G_i(x)-s_{\idxVC{i}}=0$ and projected afterwards onto \( C^\mathrm{s}\cap C_0(i) \) for $\signature_i=0$
to detect potential bi-active points and possible switches. 
The new value for $y$ is given by $\tilde{\hat{y}} - \tilde{w}$ and for the missing values of the VC in the lower branch with zero.
We go on to prove the correctness of this modification:
\begin{theorem}
	A strong stationary point of \eqref{eq:mpvcs} is equivalent to an equilibrium point
  under the proposed strategy based on piecewise flows.
\end{theorem}

\begin{proof}
Consider to the vertical formulation given in~\eqref{eq:mpvcs}. Subsequent NLPs of the piecewise flow are solved with 
\eqref{eq: Red-NLP} and the described updating strategy. 
Let $(x^*,s^*, y^*)$ be a strong stationary point of~\eqref{eq:mpvcs}. We have 
\begin{equation*}
  \begin{aligned}
    -\nabla_{x,s} \Lag_{\MPVCS} (x^*,s^*, y^*) &\in T^{-}(C\times C^\mathrm{s}, (x^{*},s^{*})), \\
    g(x^*) &= 0, \\
    H_i(x^*) - s^{*}_{\idxCC{i}} &= 0,\quad \text{for } i=1,\ldots,l  \\
    G_i(x^*) - s^{*}_{\idxVC{i}} &= 0,\quad \text{for } i=1,\ldots,l.
  \end{aligned}
\end{equation*} We get a signature $\signature$ as described in the branch initialization section from vanishing and controlling constraints and their slack variables.
It follows $s^* \in C^\mathrm{b}(\signature) \subseteq C^\mathrm{s}$ and 
since $-\nabla_{x,s} \Lag_{\MPVCS}^{\rho} (x^*,s^*, y^*) = -\nabla_{x,s} \Lag_{\MPVCS} (x^*,s^*, y^*)$ we get
\begin{equation*}
  \begin{aligned}
    0 &= P_{T(C\times C^\mathrm{b}(\signature), (x*,s*))} \left( -\nabla_{x,s} \Lag^{\rho} (x*, s*, y*) \right) = (\dot{x}(t),\dot{s}(t)),\\
    0 &= \nabla_{y} \Lag_{\MPVCS} (x^*,s^*, y^*) = \nabla_{y} \Lag_{\MPVCS}^{\rho} (x^*,s^*, y^*) = \dot{y}(t).    
  \end{aligned}
\end{equation*} 

Conversely, assume an equilibrium point $(x^*,s^*, y^*)$, i.e., \( \dot{x}(t) = 0, \dot{s}(t)=0 \) and \( \dot{y}(t) = 0 \). Then,
\begin{align*}
	0 = \dot{y}(t) = 
	\begin{pmatrix*}[l]
		\multicolumn{2}{c}{g(x^*)}\\
    H_i(x^*) - s^{*}_{\idxCC{i}} = 0 & \text{for } i=1,\ldots,l  \\
    G_i(x^*) - s^{*}_{\idxVC{i}} = 0 & \text{for } i=1,\ldots,l.
	\end{pmatrix*}
\end{align*}
and therefore $\nabla_{x,s} \Lag_{\MPVCS}^{\rho} (x^*,y^*, y^*) = \nabla_{x,s} \Lag_{\MPVCS}(x^*,s^*, y^*)$.
From \( \dot{x}(t) = 0 \), we obtain \( -\nabla_{x,s} \Lag_{\MPVCS} (x^*,s^*, y^*) \in T^{-}(C\times C^\mathrm{b}(\signature), (x^{*},s^{*}))\) 
with signature $\signature$, which was used to solve the last~\eqref{eq: Red-NLP}. By construction we have $C^\mathrm{b}(\signature)\subseteq C^\mathrm{s}$.
We check the conditions on the Lagrange multipliers for satisfying strong stationarity. Depending on $\signature_i$ for $i=1,\dots,l$ we have the following cases:\begin{itemize}
  \item If $\signature_i = 0$ and $s_{\idxVC{i}} <0 $ we have $i\in\IndexSet_{0-}$. By our strategy 
  the respective Lagrange multiplier for $s_{\idxVC{i}}$ is zero and for the slack of the controlling constraint $s_{\idxCC{i}}$ free, 
  since the controlling constraint is treated as an equality constraint. 
  \item If $\signature_i = 0$ and $s_{\idxVC{i}}=0$, we have $i\in\IndexSet_{00}$. Again the respective Lagrange multiplier for $s_{\idxVC{i}}$ is zero, 
  but for the slack of the controlling constraint $s_{\idxCC{i}}$ is less than or equal to zero. If the Lagrange multiplier for $s_{\idxCC{i}}$ would be positive we
  would not terminate, since the controlling constraint becomes inactive and a switch must performed to the upper branch.
  \item If $\signature_i = 1$ and $s_{\idxCC{i}}=s_{\idxVC{i}}=0$, we have $i\in\IndexSet_{00}$. 
  The Lagrange multiplier for the slack of the vanishing constraint is zero and for the 
  slack of the controlling constraint $s_{\idxCC{i}}$ is less than or equal to zero, 
  otherwise we would not terminate at the biactive point. 
  \item The remaining index sets are $\IndexSet_{0+},\,\IndexSet_{+0}$, and $\IndexSet_{++}$, which are covered by $\signature_i=1$. 
  The Lagrange multipliers match the strong stationarity conditions. 
\end{itemize}
It follows that \( -\nabla_{x,s} \Lag_{\MPVCS} (x^*,s^*, y^*) \in T^{-}(C\times C^\mathrm{s}, (x^{*},s^{*}))\).
\end{proof}

\section{Numerical Experiments}\label{sec:num_ex}

In the following we will demonstrate the effectiveness of our approach
on several test instances of MPVCs. While large suites of MPEC
benchmark problems, such as \texttt{NOSBENCH}~\cite{solving_mpccs} and
\texttt{MacMPEC}~\cite{MacMPEC} are available, there are no larger suites of MPVC benchmark problems, but several problems are
available in~\cite[Sec. 9.5]{diss_hoheisel}, which we adapt for our purposes. 
We convert all of these problems to their vertical form and initialize our methods 
at $x_{0} = 0$,$ y_{0} = 0$ unless stated otherwise.
All experiments were conducted on an \texttt{Intel Core i7-14700} clocked
at \SI{2.1}{\giga\hertz} with \SI{64}{\giga\byte} RAM based on an implementation in \texttt{Python 3.13.5} with \texttt{CasADi 3.7.1}~\cite{casadi} and
using \texttt{Ipopt 3.14.11}~\cite{ipopt} as subproblem solver for solving \eqref{eq: Red-NLP}.

\subsection{An Academic Example}
We begin by examining the following small academic example in two variables:
\begin{equation*}
  \begin{aligned}
    \min_{x \in \Real^{2}} \quad & 4x_{1} + 2x_{2} \\
    \st \quad &  x_{1} \geq 0 \\
                                 & x_{2} \geq 0 \\
    & \left( x_{1} + x_{2} - 5 \sqrt{2}  \right) x_{1} \geq 0 \\
    & \left( x_{1} + x_{2} - 5 \right) x_{2} \geq 0
  \end{aligned}
\end{equation*}
This example has the advantage of being low-dimensional and therefore
easy to reason about. As shown in Figure~\ref{fig:num_ex:academic_flow}, the problem
consists of a polyhedron with an attached line segment as well as a
single isolated point. It has a local optimum at ${\tilde{x} \define (0, 5)}$ at
the tip of the line segment as well as a global optimum at ${x^{*} \define (0, 0)}$ in
the isolated point. \texttt{Ipopt} for example finds with a suitable initial guess one of these points, 
but usually converging to the point ${\hat{x} \define (5\sqrt{2},0)}$, a corner of the 
polyhedron where the second controlling constraint turns active.
Numerical results of our method are shown in Figure \ref{fig:num_ex:academic}.
We visualize a flow in Figure \ref{fig:num_ex:academic_flow} starting with $\lambda=100$ and 
slightly decreasing by a factor of $1.01$. The flow consists of 637 points computed with 4602 \texttt{Ipopt} iterations.
Note, that we use $\lambda_\mathrm{init}=0.1$ and a decreasing factor of 
$\lambda_\mathrm{inc}=2.1$ from now on. This results in a convergence in 
$6$ iterations with $37$ \texttt{Ipopt} iterations of the academic example. 
Due to the slack variable reformulation a mismatch of the constraints 
during the solution process depending on the starting point can be observed.
Furthermore our method converges from every starting point inside the test 
region $[\num{-3.75},\num{11.75}]\times[\num{-3.75},\num{11.75}]$ to a 
minimum (see Figure \ref{fig:num_ex:academic_regions}).
The results are comparable to the global relaxation approach in e.g. \cite{Hoheisel22}. 
Due to the branch initialization process, the regions of convergence
differ from the results in \cite{Hoheisel22}.
If we put every infeasible pair of controlling and vanishing constraint in the lower branch $C_0$, 
the region of convergence of the global minimum would expand close to 
the local minimum, but not reaching the feasible set beside the 
isolated global minimum (see Figure \ref{fig:num_ex:academic_regions_2}).
\begin{figure}[ht]
        \centering\includegraphics[width=0.4\textwidth]{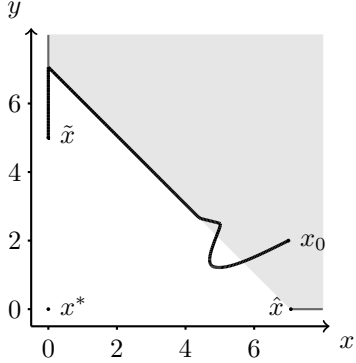}
        \caption{One example flow for the starting point (7,2) is shown in black. 
         The feasible set is marked in grey with active controlling constraints in dark grey, showing the nonconvex structure.
         The global optimal point is $x^*$, the local optimal point is $\tilde{x}$, and a further convergence point for standard solver is $\hat{x}$.}   
    \label{fig:num_ex:academic_flow}
\end{figure}
\begin{figure}[ht]
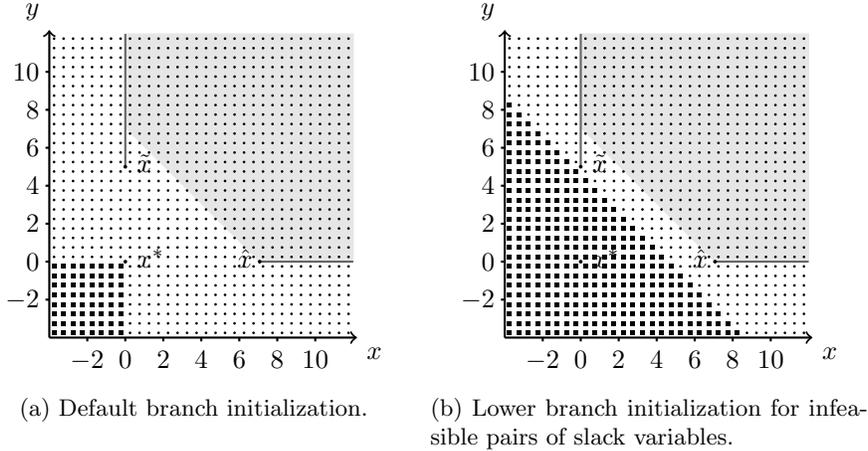

    \begin{subfigure}[t]{0.48\textwidth}
        \centering\includegraphics[width=0.9\textwidth]{academic_regions}
        \caption{Default branch initialization.}
        \label{fig:num_ex:academic_regions}
    \end{subfigure}~
    \begin{subfigure}[t]{0.48\textwidth}
        \centering\includegraphics[width=0.9\textwidth]{academic_regions_2}
        \caption{Lower branch initialization for infeasible pairs of slack variables.}
        \label{fig:num_ex:academic_regions_2}
    \end{subfigure}
    \caption{Numerical results for the academic example. The feasible set is marked in grey with active controlling constraints in dark grey.
    The global optimal point is $x^*$, the local optimal point is $\tilde{x}$, and a further convergence point for standard solver is $\hat{x}$.
    Starting points where our method converges to the local optimum is shown with circles. Rectangles mark the convergence of our method to the global minimum.
    Figure (a) shows the convergence regions with the prescribed branch initialization strategy.
    Figure (b) shows the convergence regions, if every infeasible pair of slack variables is initialized in the lower branch.}
    \label{fig:num_ex:academic}
\end{figure}
%

%
\subsection{Truss Topology Optimization}
Truss topology optimization problems appear in civil engineering
application, where structures are supposed to be designed in such a
way as to resist certain external loads without buckling while being
attached at certain points and free to flex under loads at others
(see~\cite{topology_optimization} for an introduction). Examples of
such structures are bridges designed to handle the load of
passing vehicles or cantilevers fixed at either end while having a
load attached at the other. Truss design problems can be stated as
\begin{equation*}
  \begin{aligned}
      \min_{a \in \Real^{N}, u \in \Real^{Ld}} \quad & \sum_{i = 1}^{N} \ell_{i} a _i \\
    \st \quad & K(a) u_{l} = f_{l} & \quad \text{for all } l = 1, \ldots , L \\
                                                     & f_{l}^{T} u_{l} \leq c & \quad \text{for all } l = 1, \ldots , L \\
                                                     & 0 \leq a_{i} \leq \overline{a} & \quad \text{for all } i = 1, \ldots , N \\
                                                     & \left( \overline{\sigma}^{2} - \sigma_{i l}(a, u)^{2} \right) a_{i} \geq 0 & \quad \text{for all } i = 1, \ldots , N, l = 1, \ldots , L. \\
  \end{aligned}
\end{equation*}
Here, the variables $a_{i}$ denote the cross-sectional area of a set
of $N > 0$ potential bars (which are actually present only when
$a_{i} > 0$) connecting a number of points in 2-dimensional space,
where each bar has a fixed length of $\ell_{i} > 0$.  Naturally, the area
is bounded below by zero and above by some constant
$\overline{a} > 0$. The objective of the problem is to reduce the
amount of material required by the truss, which is proportional to the
sum of volumes over all bars.

As mentioned above, the designed structure is supposed to resist
external loads. Specifically, there are $L \in \Nat$ load cases,
during each of which a load exerts forces on a subset of the
points. The points themselves fall into two categories: They are
either fixed points at which the structure is attached, or they are
free points at which a load case may cause a non-zero
displacement. The displacements of the latter points under different
load cases are contained in the variables
$u = (u_{1}, \ldots, u_{L})$, where we let $d \in \Nat$ be the total
dimension of each $u_{l}$, given by the product of the number of free
points and the ambient dimension of $2$. The matrix
\begin{equation*}
  K(a) \define \sum_{i = 1}^{N} a_{i} \frac{E}{\ell_{i}} \gamma_{i} \gamma_{i}^{T}
\end{equation*}
is the global stiffness matrix of the problem with a fixed value $E$
for Young's modulus and vectors $\gamma_{i} \in \Real^{d}$ containing
the cosines of the angles between the displacement coordinate and bar
axes. The equations $K(a) u_{l} = f_{l}$ then model equilibrium
conditions, where each vector $f_{l} \in \Real^{d}$ contains the
external loading force applied to the free points, while the
inequalities $f_{l}^{T} u_{l} \leq c$ bound the compliance of
the structure by another constant $c > 0$.  Finally, the structure
must obey the \emph{stress constraint}, stating that the absolute
value of the stress
\begin{equation*}
  \sigma_{i l}(a, u) \define E \frac{\gamma_{i}^{T} u_{l}}{\ell_i}
\end{equation*}
should be bounded above by a constant $\overline{\sigma} > 0$. This
condition is squared to remove the absolute value. The constraint is
however only required when a bar is actually present (\ie $a_{i} > 0$)
and redundant otherwise. Hence, it is formulated as a vanishing
constraint controlled by the value of $a_{i}$. We would
like to point out all of our instances are normalized
by setting $E \equiv 1$ and $\|f_{\ell}\|_{2} = 1$.
Problems are listed in Table~\ref{tab:num_ex:problemslist} 
with their solutions in Table~\ref{tab:num_ex:problemssol}.

\begin{table}[hb]
  \caption{List of problem instances with sizes and parameters. Note that the number of VCs corresponds to the number of bars.}
  \label{tab:num_ex:problemslist}
  \centering\begin{tabular}{l S[table-format=4] S[table-format=3] S[table-format=3] S[table-format=3]  S[table-format=3] S[table-format=3.1]}
     \toprule
     \textbf{Problem} & \textbf{$\#$Vars} & \textbf{$\#$Cons} &\textbf{$\#$VC}  & $\bar{a}$ & $c$ & $\bar{\sigma}$ \\
     \midrule
    \instName{TenBar}&   29 &  19 &  10 & 100 &  10 &   1.0\\
    \instName{Cant1} &  497 & 273 & 224 &   1 & 100 & 100.0\\
    \instName{Cant2} &  497 & 273 & 224 &   1 & 100 &   2.2\\
    \instName{Hook1} & 1417 & 756 & 661 & 100 & 100 & 100.0\\
    \instName{Hook2} & 1417 & 756 & 661 & 100 & 100 &   3.5\\
    \instName{Hook3} & 1417 & 756 & 661 & 100 & 100 &   3.0\\
    \bottomrule
  \end{tabular}
\end{table}

\Paragraph{Initialization} Due to the large number of vanishing
constraints being present in truss topology problems, an
initialization of zero places the initial point of the flow in a
position where all vanishing constraints are bi-active while being far
away from stationarity, causing the flow to enter a branch constrained
by many variables fixed to zero. We find it to be more advantageous to
instead start the optimization at a point bounded away from zero in
the upper branch of the problem. To this end, we set
$a_i \equiv \alpha$ for all $i$ with a constant $\alpha > 0$. Based on
this choice of $\alpha$, we want $u_l$ to be close to
$\argmin_{u \in \Real^{d}}\|K(a) u - f_{l}\|$ in order to reduce the
residual of the corresponding constraint. In our examples, $K(1)$ is
always invertible and we can choose $u_{l} \define \alpha^{-1} K(1)^{-1} f_{l}$.
To ensure that the compliance bound is met, we need that
$\alpha \geq c^{-1} f_{l}^{T}K(1)^{-1}f_{l}$ for all load cases.
Similarly, for the stress constraint, we need that
\begin{equation*}
  \frac{\gamma_i^{T} K(1)^{-1} f_{l}}{\ell_{i} \overline{\sigma}} \leq \alpha
  \quad \text{ for all } i = 1, \ldots, N, l = 1, \ldots, L.
\end{equation*}
By taking the maximum of all these bounds we obtain an $\alpha$ such
that the resulting values of $a$ and $u$ yield a solution in the upper branch, 
which is feasible under the mild assumption that the resulting values of $a$ are 
all bounded above by $\overline{a}$.

To compute a dual estimation based on a given primal for a given
instance of~\eqref{eq:mpvcs}, we conduct the following steps: First,
we compute the signature $\signature$ compatible with the primal
solution. Second, we consider the corresponding
problem~\eqref{eq:branch_nlp} and compute a dual solution minimizing its
KKT residuum $\|f(x) + (y, y_{s})^{T} J_{g}(x, s) \|_{2}$ while
maintaining bounds on $y_{s}$ compatible with the active set, \ie
forcing entries of $y_{s}$ to be equal to zero for inactive variables
and to be non-negative or non-positive, depending on the active
bounds, otherwise. The resulting box-constrained quadratic program can
be quickly solved using standard methods, yielding a reasonable initial
dual solution $y$.

\Paragraph{Algorithmic settings}
The parameters $\rho=10^{-2}$, $\lambda_\mathrm{init}=0.1$, $\lambda_\mathrm{inc}=2.1$ 
were selected for the calculation of all results. 
Unless otherwise specified, the variables are scaled before using \eqref{eq: Red-NLP}, 
so that all primal and dual variables with an absolute value greater than $1$ are scaled 
to unity in every iteration and the scalar products are not scaled.

\subsubsection{Ten-bar Truss}
As the first truss design problem we consider the so-called
\emph{ten-bar} problem from~\cite{diss_hoheisel} shown
in Figure~\ref{fig:num_ex:tenbar_potential}. It consists of
ten potential bars to be placed between six points, two
of which being fixed to a wall and a loading force
pointing downwards. We set the parameters to $(\bar{a}, c, \bar{\sigma}) = (\num{100.0}, \num{10.0}, \num{1.0})$.

Based on the primal-dual initialization
given above, our method converges in $7$ iterations with $181$ \texttt{Ipopt} iterations within less than a second with an
optimal volume $V^{*} = 8.0$. The resulting solution is
shown in Figure~\ref{fig:num_ex:tenbar_solution} with individual solution
values being shown in Table~\ref{tab:num_ex:tenbar}.
Note that there is no active compliance constraint, but several feasible stress constraints $\sigma_i^2\leq\bar{\sigma}^2$, 
which differs from the solution structure seen in \cite{diss_hoheisel}, 
while the optimal volume of our solution does coincide with that in \cite{Hoheisel22}.
The value of the compliance constraint and the values of the $a_i$ (with reordering) coincide with the results in \cite{benko_sqp}.
\begin{figure}[ht]
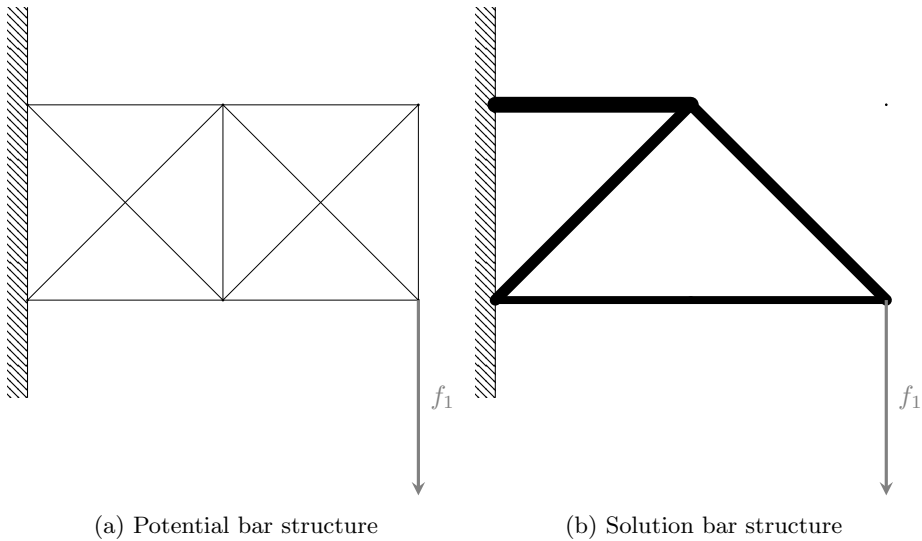

    \begin{subfigure}[t]{0.5\textwidth}
        \includegraphics[width=\textwidth]{tenbar_potential}
        \caption{Potential bar structure}
        \label{fig:num_ex:tenbar_potential}
      \end{subfigure}~
    \begin{subfigure}[t]{0.5\textwidth}
        \includegraphics[width=\textwidth]{tenbar}
        \caption{Solution bar structure}
        \label{fig:num_ex:tenbar_solution}
      \end{subfigure}
    \caption{Potential bars and solution of the \instName{TenBar} instance, which visually coincide with e.g. \cite{diss_hoheisel,Hoheisel22}.
    The load force $f_1$ is marked in dark grey. On the left the structure is fixed on the grey wall.}
    \label{fig:num_ex:tenbar}
\end{figure}
\begin{table}
  \caption{Results of the \instName{TenBar} instance for the controlling constraints $a_i$, the stress values $\sigma_i$, the displacement variables $u_j$.
  We also show the compliance $f_1^Tu$ and optimized Volume $V^*$}
  \label{tab:num_ex:tenbar}
  \centering\begin{tabular}{S[table-format=2] S[table-format=1.6] S[table-format=1.6] | S[table-format=2] S[table-format=1.6]}
     \toprule
     {$i$} & {$a_i^*$} & {$\sigma_i(a^*,u^*)$} & {$j$} & {$ u_j^*$}\\
     \midrule
     1 & 1.00000 & -1.00000 &1& -1.00000\\
     2 & 1.41421 & -1.00000 &2& -2.72241\\
     3 & 2.00000 &  1.00000 &3&  1.00000\\
     4 & 0       &  0.86120 &4& -3.00000\\
     5 & 0       & -0.27758 &5& -2.00000\\
     6 & 1.00000 & -1.00000 &6& -8.00000\\
     7 & 0       & -0.91333 &7&  1.47111\\
     8 & 0       &  0.47111 &8& -7.02019\\
     9 & 1.41421 &  1.00000 &\multicolumn{2}{c}{$f^Tu = \num{8.00000} $} \\
     10 & 0      &  0.97980 &\multicolumn{2}{c}{$V^* = \num{8.00000}$} \\
     \bottomrule
  \end{tabular}
\end{table}

\FloatBarrier

\subsubsection{Cantilever Problem}

The next, more challenging truss design problem consists of the design
of a cantilever arm. The instances, shown in
Figure~\ref{fig:num_ex:cantilever_potential}, consist of \num{27} points
placed in three rows. Once again, the leftmost points are attached
to a wall, while a loading force pointing downwards is applied
to the rightmost node on the bottom. Potential bars are present
between all pairs of points where long bars overlapping shorter
ones are removed. To this end, we let $(x_{i}, y_{i}) \in \Nat^{2}$ be (integer) coordinates of the points, noting that the bar between points $i$ and $j$
is overlapping if $\gcd(x_i - x_j, y_i - y_j)\neq 1$. Consequently,
potential bars are present for \num{224} of the \num{351} pairs of points.
Following \cite{diss_hoheisel}, We consider three parameter sets $\bar{a} = 100,~c=1.0,~\bar{\sigma}\in\{2.2,100\}$, yielding the instances \instName{Cant1} and
\instName{Cant2}.

The solutions, depicted in Table~\ref{tab:num_ex:problemssol} are
similar to those of the previous instance, where convergence is
achieved within less than $20$ iterations and $1287$ to $3228$ \texttt{Ipopt} Iterations. The computation
time however differs widely between the instances, depending on the choice of scaling.  As can be seen from the final
configuration, several vanishing constraints become active in
\instName{Cant2a} and \instName{Cant2b}, and several branches have to be explored until a
solution is found, increasing the number of iterations, but not necessarily run times. 
Conversely, for \instName{Cant1} no VC is in the lower branches at
the solution, coinciding with the results in~\cite{diss_hoheisel}.

In Figure~\ref{fig:num_ex:cantilever2} we also present two solutions for instance \instName{Cant2} in order to show that small differences in the method
can lead to a different local solution. We therefore divide the 
instance \instName{Cant2} into two subinstances. In instance \instName{Cant2a}, we use the algorithmic parameters as they are. 
In instance \instName{Cant2b}, the scaling of the variables is turned off and, at the same time, 
the scalar products for the primal variables and slack variables as described in \eqref{eq:scaled_nlp} are changed from the Euclidean scalar products
to the scalar products with absolute values of the primal Lagrange gradient of \eqref{eq:mpvcs} on the diagonal and added value of $\rho$.
The solutions of \instName{Cant2a} and \instName{Cant2b} differ and is a typical behaviour of vanishing constraint
problems caused by their non-convex feasible set. 
Although fewer bars are required for the truss structure, the volume of $30.63$ in \instName{Cant2b} is higher than the volume of $23.9741$ in \instName{Cant2a}. 
At the same time, more \texttt{Ipopt} iterations are required, which leads to a longer runtime for instance \instName{Cant2b}.
The results are noted in Table \ref{tab:num_ex:problemssol}.

\begin{figure}[ht]
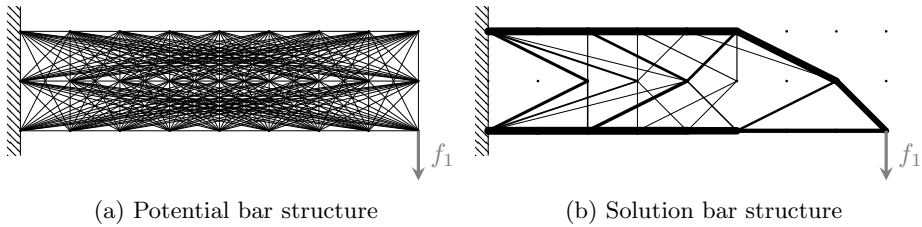

    \begin{subfigure}[t]{0.5\textwidth}
        \includegraphics[width=\textwidth]{cantilever_potential}
        \caption{Potential bar structure}
        \label{fig:num_ex:cantilever_potential}
      \end{subfigure}~
    \begin{subfigure}[t]{0.5\textwidth}
        \includegraphics[width=\textwidth]{cantilever_1}
        \caption{Solution bar structure}
        \label{fig:num_ex:cantilever_solution_1}
      \end{subfigure}
      \caption{Potential structure of the cantilever problem in (a) and the solution of the \instName{Cant1} instance in (b). 
      The load force $f_1$ is marked in dark grey. The structure is fixed to the wall on the left-hand side.}
    \label{fig:num_ex:cantilever}
\end{figure}
\begin{figure}[ht]
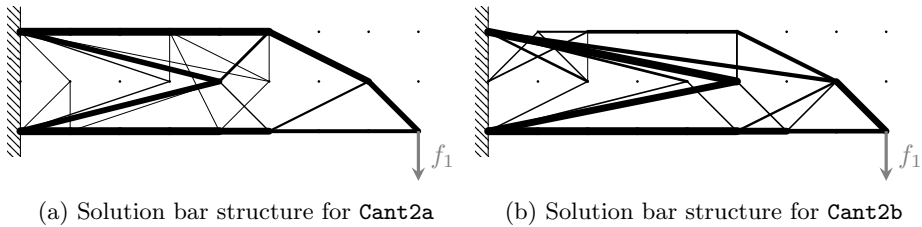

    \begin{subfigure}[t]{0.5\textwidth}
        \includegraphics[width=\textwidth]{cantilever_2}
        \caption{Solution bar structure for \instName{Cant2a}}
        \label{fig:num_ex:cantilever_solution_2}
      \end{subfigure}~
    \begin{subfigure}[t]{0.5\textwidth}
        \includegraphics[width=\textwidth]{cantilever_22}
        \caption{Solution bar structure for \instName{Cant2b}}
        \label{fig:num_ex:cantilever_solution_22}
      \end{subfigure}
      \caption{Influence of scaling on instance \instName{Cant2}:
       Scaling of variables in (a) and not in (b), without scalar product scaling in (a)
       and primal and slack variable scalar product scaling based on the Lagrange gradient and parameter $\rho$ in (b).}
    \label{fig:num_ex:cantilever2}
\end{figure}

\FloatBarrier
\subsubsection{Hook Problem}

The final problem, again described in \cite{diss_hoheisel}, is called
the \emph{hook problem} owing to the shape of the structure as
depicted in Figure~\ref{fig:num_ex:hook_potential} consisting of
\num{35} nodes. The structure is attached at its top with a load being
applied to a node at its far right.  All nodes once more have integer
coordinates, enabling us to detect overlapping potential bars as
before, resulting in a total of~\num{661} potential bars.
We consider three parameter sets $\bar{a} = 100,~c=100,~\bar{\sigma}\in\{3,3.5,100\}$.

The detailed solutions are listed in Table~\ref{tab:num_ex:problemssol} and visualized in Figure~\ref{fig:num_ex:hook}. 
In \instName{Hook1}, the first parameter set with $\bar{\sigma}=100$ is used. It can be observed that, as in \cite{diss_hoheisel},
the stress constraint is not active. The structure forms a half-circle open at the top (see Figure~\ref{fig:num_ex:hook_solution_1}).
With 21 bars, there are significantly fewer bars than in \cite{diss_hoheisel}, which can be related to the single load case. 

The problem instance \instName{Hook2} attains the maximal stress of $\bar{\sigma}=3.5$, which turns active during the 
optimization but not in the solution, leading to the
same structure with the same volume of $9.1118$. These results of \instName{Hook1} and \instName{Hook2} 
illustrate the influence of the parameter $\bar{\sigma}=3.5$, as reflected in the shorter run time, 
different maximal stress, and iteration numbers. In \instName{Hook3} we further decrease the $\bar{\sigma}$ to $3.0$.
Again, the solution bar structure consists of $21$ bars showing the same structure as \instName{Hook1}, but, the maximal stress value for 
solution bars $\hat{\sigma}_{\max}=3.0$ is active. The volume has increased to $10.0691$. The number of vanished constraints again increases 
with the decreasing of $\bar{\sigma}$, i.e., we have constraints in the lower branch $C_0$. 
The shorter run time is surprising and is a result of the lower number of iterations of both our approach 
and the subproblem solver \texttt{Ipopt}.

Lastly, owing to the larger problem sizes, the computation time increases significantly approximately one to three hours 
for the instances compared to the cantilever instances. 
However, the number of iterations of our approach remains roughly small. 
Thus, progress towards convergence is impeded by the increase in time required to solve the individual subproblems.
\begin{table}[ht]
  \caption{Detailed solutions of listed problem from Table~\ref{tab:num_ex:problemslist} showing the optimal volume $V^*$, the numbers of resulting bars,
    the maximal stress values $\sigma_{\max}$ and the maximal stress value for solution bars $\hat{\sigma}_{\max}$ (i.e., $a_i>0$), the configuration between lower branches
  ($\#C_0$) and upper branches ($\#C_+$), the number of iterations, the overall number of \text{Ipopt} iterations, and the computation time in seconds.}
  \label{tab:num_ex:problemssol}
  \begin{adjustbox}{max width=\textwidth}
    \centering\begin{tabular}{l S[table-format=2.4] S[table-format=2] S[table-format=2.4] S[table-format=2.4]  r S[table-format=2] S[table-format=4] S[table-format=5.2]}
      \toprule
      \textbf{Problem} & \textbf{$V^*$} & \textbf{\#bars} & $\sigma_{\max}$ & $\hat{\sigma}_{\max}$ &  $\#C_0 | \#C_+$ & \textbf{\#it} & \textbf{\texttt{ipopt}\#it} & {\textbf{time} [\si{\second}]}\\
      \midrule
      \instName{TenBar}&  8.0000 &  5  & 1.0     & 1.0     & $0|10$\,~~~~&  7 &  181 &    0.40 \\
      \instName{Cant1} & 23.1399 & 37  & 2.78132 & 2.78132 & $0|224$~~  & 13 & 1287 &  395.08 \\
      \instName{Cant2a}& 23.9741 & 32  & 3.08450 & 2.20000 & $18|206$~~ & 14 & 1013 &  318.86 \\
      \instName{Cant2b}& 30.6300 & 30  & 7.39528 & 1.81591 & $2|222$~~ & 18 & 3228 &  949.15 \\
      \instName{Hook1} &  9.1118 & 21  & 7.21085 & 3.31283 & $0|661$~~  & 14 & 2968 & 11237.90 \\
      \instName{Hook2} &  9.1118 & 21  & 3.49038 & 3.31282 & $0|661$~~  & 18 & 1767 & 6928.56 \\
      \instName{Hook3} & 10.0619 & 21  & 5.19010 & 3.00000 & $8|653$~~ & 12 &  787 & 3288.61 \\
      \bottomrule
    \end{tabular}
  \end{adjustbox}
\end{table}

\begin{figure}[ht]
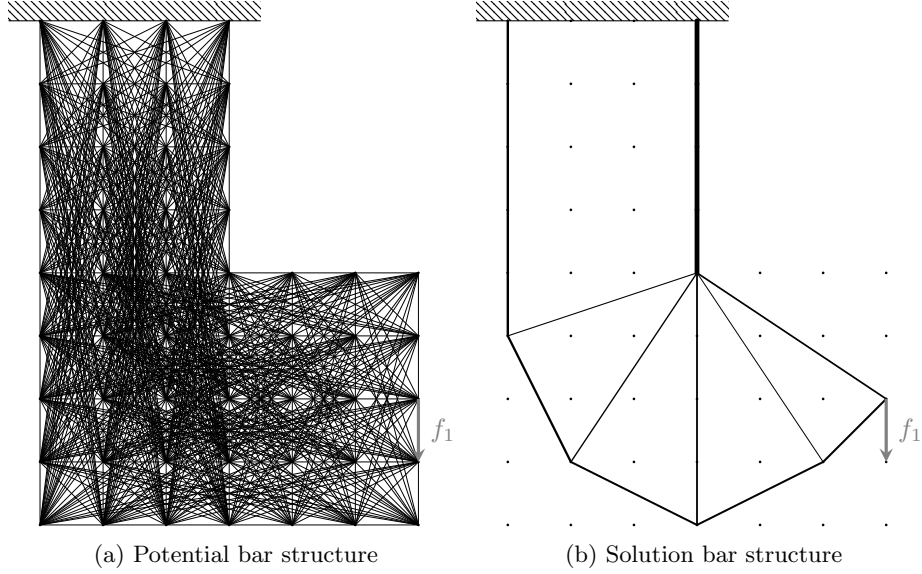

    \begin{subfigure}[t]{0.5\textwidth}
        \includegraphics[width=\textwidth]{hook_potential}
        \caption{Potential bar structure}
        \label{fig:num_ex:hook_potential}
      \end{subfigure}~
    \begin{subfigure}[t]{0.5\textwidth}
        \includegraphics[width=\textwidth]{hook_1}
        \caption{Solution bar structure}
        \label{fig:num_ex:hook_solution_1}
      \end{subfigure}
    \caption{Potential and solution bar structure of \instName{Hook1}, \instName{Hook2}, and \instName{Hook3}}
    \label{fig:num_ex:hook}
\end{figure}

\subsubsection{Comparison to Relaxation Approach}
We compare the results with the calculations of the relaxation approach from~\cite{diss_hoheisel}. 
For truss topology design problems, this is: \begin{equation*}
  \begin{aligned}
      \min_{a \in \Real^{N}, u \in \Real^{Ld}} \quad & \sum_{i = 1}^{N} \ell_{i} a _i \\
    \st \quad & K(a) u_{l} = f_{l} & \quad \text{for all } l = 1, \ldots , L \\
                                                     & f_{l}^{T} u_{l} \leq c & \quad \text{for all } l = 1, \ldots , L \\
                                                     & 0 \leq a_{i} \leq \overline{a} & \quad \text{for all } i = 1, \ldots , N \\
                                                     & r^t_{il}(a,u) \geq -t & \quad \text{for all } i = 1, \ldots , N, l = 1, \ldots , L. 
  \end{aligned}
\end{equation*} Where $t$ is the \textit{relaxation parameter} and \begin{equation*}
  \begin{aligned}
    r^t_{il}(a,u) &\define \varphi^t(a_i,\overline{\sigma}^{2} - \sigma_{i l}(a, u)^{2} ) \quad \text{for all } i = 1, \ldots , N, l = 1, \ldots , L,~ \text{with}\\
    \varphi^t(w,v) &\define \frac{1}{2}\left( wv + \sqrt{w^2v^2+t^2} + \sqrt{v^2+t^2} -v\right).
  \end{aligned}
\end{equation*}
The results are presented in Table~\ref{tab:num_ex:relaxsol}. The computation times whenever a solution is obtained was around ten times faster.
Consistent with the observations in~\cite{diss_hoheisel}, we achieve convergence for the instances \instName{TenBar} and \instName{Cant1} 
using the relaxation parameter of $t = 10^{-2}$. Note the importance of selecting an appropriate value of $t$, otherwise \texttt{Ipopt} 
aborts with maximal iterations reached or restoration failed messages (marked with DNF in Table~\ref{tab:num_ex:relaxsol}). 
For instance, in \instName{hook1} it can be observed that convergence is not achieved, even when different relaxation parameters are applied.
Figure~\ref{fig:num_ex:relax} shows representative outputs from the relaxation approach.

Figure~\ref{fig:num_ex:cantilever_relax} displays the optimal solution of \instName{Cant2} for a relaxation parameter of $t = 10^{-3}$,
which is structurally consistent with our flow approach solution for \instName{Cant1} and the results from~\cite{diss_hoheisel}.
The volume is slightly smaller with a value of $22.8013$ compared to all cantilever instances computed with our approach. 
The overall iterations of \texttt{Ipopt} are less in the relaxation approach.
Figures~\ref{fig:num_ex:hook_relax_1} and \ref{fig:num_ex:hook_relax_2} depict the outputs for \instName{Hook1} obtained with relaxation parameters
$t = 10^{-2}$ and $t = 10^{-5}$, respectively. The first output is a feasible structure similar to our gradient-flow solution 
of the hook instances, but with a larger volume of $9.19915$ and $18$ bars. 
The second output is an infeasible point and shows how different the last iteration of the solver can look like.

Unfortunately, we were unable to reproduce the same local minima found in~\cite{diss_hoheisel}.
Due to the non-convexity of vanishing constraint problems, it is not surprising that the results are not always fully correspond with previous calculations. 
In contrast, our gradient-flow approach successfully solves all listed instances. However, if only an approximation 
is required and a suitable relaxation parameter exists, the relaxation approach can save 
a considerable amount of computation time, but having the risk only to converge to feasible points.

\begin{table}[hb]
  \caption{Detailed solutions of the relaxation approach with varying relaxation parameter $t$. We denote the iteration number of \texttt{Ipopt}~\cite{ipopt}, if it the solver printed the message 'Optimal solution found'. 
  Otherwise we write \text{DNF} as short for did not finish.}
  \label{tab:num_ex:relaxsol}
  \begin{adjustbox}{max width=\textwidth}
    \centering\begin{tabular}{l S[table-format=4] S[table-format=4] S[table-format=4] S[table-format=4] S[table-format=4] S[table-format=4] S[table-format=4] S[table-format=4] }
      \toprule
      \textbf{Value of $t$} & \textbf{$10^{-2}$} & \textbf{$10^{-3}$} & \textbf{$10^{-4}$} & \textbf{$10^{-5}$} & \textbf{$10^{-6}$} & \textbf{$10^{-7}$} & \textbf{$10^{-8}$} & \textbf{$10^{-9}$} \\
      \midrule
      \instName{TenBar}& 34 & 156 & \text{DNF} & \text{DNF} & 46 & 66 & 53 & \text{DNF} \\
      \instName{Cant1} & 690 & 982 & \text{DNF} & \text{DNF} & \text{DNF} & \text{DNF} & \text{DNF} & \text{DNF} \\
      \instName{Cant2} & \text{DNF} & 571 & 541 & \text{DNF} & \text{DNF} & \text{DNF} & \text{DNF} & \text{DNF} \\
      \instName{Hook1} & \text{DNF} & \text{DNF} & \text{DNF} & \text{DNF} & \text{DNF} & \text{DNF} & \text{DNF} & \text{DNF} \\
      \bottomrule
    \end{tabular}
  \end{adjustbox}
\end{table}

\begin{figure}[ht]
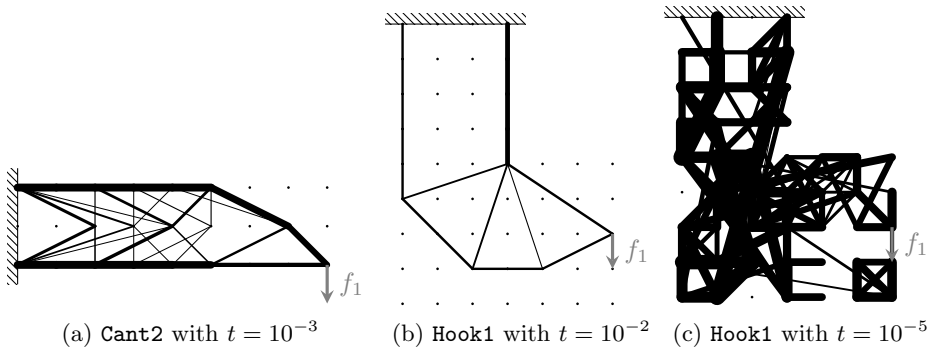

    \begin{subfigure}[t]{0.4\textwidth}
        \includegraphics[width=\textwidth]{cantilever_relax}
        \caption{\instName{Cant2} with $t=10^{-3}$}
        \label{fig:num_ex:cantilever_relax}
      \end{subfigure}~
    \begin{subfigure}[t]{0.3\textwidth}
        \includegraphics[width=\textwidth]{hook_relax_1}
        \caption{\instName{Hook1} with $t=10^{-2}$}
        \label{fig:num_ex:hook_relax_1}
      \end{subfigure}~\begin{subfigure}[t]{0.3\textwidth}
        \includegraphics[width=\textwidth]{hook_relax_2}
        \caption{\instName{Hook1} with $t=10^{-5}$}
        \label{fig:num_ex:hook_relax_2}
      \end{subfigure}~
    \caption{Different bar structures of \instName{Cant2} and \instName{Hook1} with different relaxation parameters. 
    In (a) we show an optimal solution. In (b) and (c) we see the output of the not finished optimizations of \instName{Hook1}.}
    \label{fig:num_ex:relax}
\end{figure}

\FloatBarrier
\section{Conclusion and Future Work}
This work offers a novel approach for computing strong stationary points in Mathematical Programs with Vanishing Constraints (MPVCs). 
The proposed method exploits the problem structure via a slack-variable formulation of vanishing and controlling constraints, 
systematically selecting and switching between convex subsets of the slack-variable feasible set.
This approach allows to keep track of the primal-dual gradient/anti-gradient flow at each iteration through the solution of a convex subproblem. 
Since these subproblems can be solved using standard optimization solvers, our approach improves practical applicability. 
Numerical experiments on truss topology design problems with a single load case demonstrate the effectiveness 
of the proposed gradient flow approach, which outperforms a compared standard relaxation method.

\section*{Acknowledgements}

Supported by the German Federal Ministry of Research, Technology and Space (BMFTR) under Grant No. 05M22VHA.

\printbibliography
